\documentclass{amsart}
\usepackage{amsmath}
\usepackage{amssymb}
\usepackage{fancybox}
\usepackage{eucal}
\usepackage{amsthm}
\usepackage{amscd}
\usepackage{hyperref}
\usepackage{wasysym}
\usepackage{url}

\usepackage{amssymb,}
\usepackage[]{amsmath, amsthm, amsfonts,graphicx, amscd,}
\usepackage[all,cmtip]{xy}
\input amssym.def \input amssym

\newtheorem {thm}{Theorem}[section]

\newtheorem*{thmstar}{Theorem}
\newtheorem{prop}[thm]{Proposition}

\newtheorem {lem}[thm]{Lemma}

\newtheorem{ques}[thm]{Question}

\theoremstyle{remark}
\newtheorem{rem}[thm]{Remark}
\newtheorem{np*}{Non-Proof}

\theoremstyle{definition}
\newtheorem{defn}[thm]{Definition}

\newtheorem{exam}[thm]{Example}

\newcommand{\pd}[2]{\frac{\partial #1}{\partial #2}}

\def\Ind{\setbox0=\hbox{$x$}\kern\wd0\hbox to 0pt{\hss$\mid$\hss} \lower.9\ht0\hbox to 0pt{\hss$\smile$\hss}\kern\wd0}
\def\Notind{\setbox0=\hbox{$x$}\kern\wd0\hbox to 0pt{\mathchardef \nn=12854\hss$\nn$\kern1.4\wd0\hss}\hbox to 0pt{\hss$\mid$\hss}\lower.9\ht0 \hbox to 0pt{\hss$\smile$\hss}\kern\wd0}
\def\ind{\mathop{\mathpalette\Ind{}}}
 
\newcommand{\m}{\mathbb }
\newcommand{\mc}{\mathcal }
\newcommand{\mf}{\mathfrak }

\title{Bertini theorems for differential algebraic geometry}
\author{James Freitag} \thanks{This material is based upon work supported by the National Science Foundation Mathematical Sciences Postdoctoral Research Fellowship, award number 1204510 and an American Mathematical Society Mathematical Research Communities Travel Grant.} 
\address{freitag@math.berkeley.edu \\
Department of Mathematics\\
University of California, Berkeley\\
970 Evans Hall\\
Berkeley, CA 94720-3840 }

\begin{document}
\begin{abstract}
We study intersection theory for differential algebraic varieties. Particularly, we study families of differential hypersurface sections of arbitrary affine differential algebraic varieties over a differential field. We prove the differential analogue of Bertini's theorem, namely that for an arbitrary geometrically irreducible differential algebraic variety which is not an algebraic curve, generic hypersurface sections are geometrically irreducible and codimension one. Surprisingly, we prove a stringer result in the case that the order of the differential hypersurface is at least one; namely that the generic differential hypersurface sections of an irreducible differential algebraic variety are irreducible and codimension one. We also calculate the Kolchin polynomials of the intersections and prove several other results regarding intersections of differential algebraic varieties. 
\\
MSC2010 classification: 03C60, 03C98, 12H05
\end{abstract}
\maketitle

\section{Introduction}
Consider the following theorem from algebraic geometry:
\begin{thmstar} \cite[7.1, page 48]{Hartshorne} \label{ZINT} Let $Y,Z$ be irreducible algebraic varieties of dimensions $r,s$ in $\m A^n$ Then every irreducible component $W$ of $Y \cap Z$ has dimension greater than or equal to $r+s-n$. 
\end{thmstar} 
This theorem fails for differential algebraic varieties embedded in affine space, as the following examples show. 
\begin{exam}\label{anomalous} (Ritt's example). We work in $\m A^3$ over an ordinary differential field, $k$. Let $V=Z(f)$, where
$$f(x,y,z)=x^5-y^5+z(x \delta y - y\delta x)^2$$
In fact, though $V$ is the zero set of an absolutely irreducible differential polynomial, it is not irreducible in the Kolchin topology. $V$ has six components.  Let $\mu_5$ denote the set of fifth roots of unity. For each $\zeta \in \mu_5,$ $x -\zeta y$ cuts out a subvariety of $V.$ 
Note that
$$f=(x -\zeta y)(\prod_{\eta \in \mu_5 \backslash \{\zeta \}} (x- \eta y) + z (\delta y (x -\zeta y) - y \delta (x -\zeta y))^2$$ is a preparation eqution for $f$ with respect to $x- \zeta y$ \cite[see chapter 4, section 13]{KolchinDAAG}. Further, one obtains the preparation congruence, $f=(x -\zeta y)(\prod_{\eta \in \mu_5 \backslash \{\zeta \}} (x- \eta y) $ modulo $[x -\zeta y]^2,$ so by the Low Power Theorem \cite[chapter 7]{Ritt} or \cite[chapter 4, section 15]{KolchinDAAG}, $[x- \zeta]$ is the ideal of a component of $V$. The general component is given by the saturation by the separant (with respect to some ranking) of $[f]$. For instance, one possible choice of ranking would yield $[f]:\pd{f}{\delta x} ^\infty = \{ g \, | \, \left( \pd{f}{\delta x} \right)^n g \in [f], \, \text{for some } n \in \m N \}$ as the ideal of the general component.  By the component theorem \cite[Theorem 5, page 185]{KolchinDAAG}, these are the only components.  Now we consider the differential algebraic variety $W$, which is the general component of $V$. Establishing that $(0,0,0) \in W$ can be done by noting that $f$ is differentially homogeneous \footnote{This is an observation of Phyllis Cassidy.} and applying \cite[Proposition 2]{KolchinDiffComp}.  In fact, one can prove that $W \cap H$ consists of precisely $ (0,0,0)$ (see \cite{SitOnlineNotes} for a nice exposition of the proof).  
\end{exam}

\begin{exam} \label{anombuium}
Let $X$ be a projective curve of genus at least two over a differentially closed field $K$ which does not descend to the constants. 
Buium \cite{buium1994geometry} proves that $X$ may be embedded as a differential algebraic variety in projective space over $K$ such that $X$ lies outside of some unique hyperplane. Passing to the affine cone of the projective space again gives an intersection with $r+s-n=1$, but the intersection has dimension $0$ (consisting of the origin). 
\end{exam}
The motivating questions of the paper come from the examples of Ritt and Buium: 
\begin{ques} 
In the space of differential hypersurfaces of a particular order and degree, what is the set of coefficients on which the intersection theorem fails for a given arbitrary differential algebraic variety? 
\end{ques} 

The main thrust of this paper is to provide an answer to this question by proving the differential algebraic analogue of Bertini's theorem. Roughly, we prove that the intersection of an irreducible differential algebraic variety of dimension $d$ and a generic hyperplane is a irreducible differential algebraic variety of dimension $d-1$.

Specifically, if $V$ is given by the differential ideal $\mf I \subseteq K \{\bar y \}$ and $ \bar u$ is a suitably long tuple of independent differential transcendentals over $K$, then we analyze the properties of the ideal $[ \mf I , f_{\bar u}] \subseteq K \langle \bar u \rangle \{y_ 1 , \ldots , y_n \}$ where $f _{\bar u} $ is the differential polynomial given by $u_0 + \sum _{i} u_i m_i (\bar y) $ where the $m_i$ are all of the differential monomials of order and degree bounded by some pair of natural numbers. This analysis was already performed in the ordinary case \cite{GDIT}. In particular, we show that $[\mf I ,f_{\bar u}]$ is a prime differential ideal of dimension one less than $\mf I$. 

This result is the differential algebraic analogue of \cite[Theorem 2 page 54]{hodge1994methods}, which is an algebraic precursor to Bertini's theorem. The main problem with applying this result for making inductive arguments in differential algebraic geometry is that the primality of the ideal holds only in the differential polynomial ring over $K \langle \bar u \rangle$. For various applications, and inductive arguments, one would wish to take the coefficients $\bar u$ in some ambient large differentially closed field and establish the primality of $[\mf I, f_{\bar u }]$ in the polynomial ring over the differentially closed field. This is too much to ask, because the corresponding theorem is not even true for algebraic varieties unless the dimension of $V$ is at least $2$; the intersection of a degree $d$ curve with a generic hyperplane consists $d$ points. Surprisingly, we show that algebraic curves are the \emph{only} obstruction to the theorem. This portion of the argument uses a differential lying-over theorem, a geometric or model theoretic argument about ranks and the algebraic results established for $[\mf I, f_{\bar u}]$. In geometric terms, we prove \emph{geometric irreducibility} results:

\begin{defn}\label{geodef} An affine differential algebraic variety, $V$ over $k$, is geometrically irreducible if $I(V/k')$ is a prime differential ideal for \emph{any} $k'$, a differential field extension of $k$.
\end{defn}

\begin{thm}\label{TheMainTheorem} ($|\Delta | =m$) Let $V$ be a geometrically irreducible affine differential algebraic variety over a $\Delta$-field $K$. Let $H$ be a generic differential hypersurface over $K$. Then $V \cap H$ is a geometrically irreducible differential algebraic variety, which is nonempty just in case $dim(V)>0$. In that case, $V \cap H$ has Kolchin polynomial: $$\omega_{V/ K }(t) -\binom{t+m-h}{m}.$$
\end{thm} 

Our analysis draws inspiration from \cite{GDIT}, but we also use arguments of a more geometric and model theoretic nature. In particular, we employ the theory of prolongation spaces, in the sense of Moosa and Scanlon \cite{MSJETS} and arguments which use several properties of Lascar rank. 

Following the proof of our main theorem, we give a result regarding intersections of differential algebraic varieties with generic differential hypersurfaces passing through a given point. Specifically, we prove that if $V$ is a $d$ dimensional differential algebraic variety and $H_1, H_2, \ldots H_{d+1}$ are generic differential hypersurfaces which contain $\bar a$, then $V \cap \bigcap_{i=1}^{d+1} H_i \neq \emptyset$ then $\bar a \in V$. 
 
The special case in which the $H_i$ are hyperplanes in the case of one derivation was proved in \cite{GDIT}; our proof of the generalization is much shorter, owing to using stability theoretic tools (e.g. Lascar's symmetry lemma). This also generalizes \cite[Theorem 1.7]{Pongembedding}. 

One expects the main theorem of paper to be useful for proofs by induction on the differential transcendence degree of a differential algebraic variety. For instance, in \cite{OmarWilliamJim}, the theorem is used to study completeness in the Kolchin topology. 


\subsection*{Acknowledgements}
I would like to gratefully acknowledge the patient and thorough explanations of Ritt's example by William Sit and Phyllis Cassidy. This example, in part, triggered my interest in the problems considered here. Thanks to Phyllis Cassidy for alerting me to the work of Xiao-Shan Gao, Wei Li, and Chun-Ming Yuan. Their techniques obviously had a major influence. 

This work constitutes a portion of my thesis, supervised by Dave Marker. Thanks to Dave for many useful conversations on the topic. Thanks also to Rahim Moosa for several useful conversations on the problems considered in this paper, several of which took place during a visit to Waterloo sponsored by an American Mathematical Society Mathematics Research Community Grant. 

\section{Setting and definitions} \label{section2}
We will very briefly review \emph{some} of the developments from model theory and differential algebra necessary for our results; more complete expositions can be found in various sources, which we cite below. We use standard model-theoretic notation, following \cite{Marker} and differential algebraic notation, following \cite{KolchinDAAG}. 
In this paper, we will take $K$ to be a differential field of characteristic zero in $m$ commuting derivations, $\Delta = \{ \delta_1 , \ldots , \delta_m \}$. In this setting, we have a \emph{model companion}, the theory of $\Delta$-closed fields, denoted $DCF_{0,m}$. One may work entirely within a saturated model $\mc U \models DCF_{0,m}$ for this paper, taking all differential field extensions therein. However, the results of this paper require care with respect to the field we work over. We do not consider abstract differential algebraic varieties or differential schemes; we only consider affine differential algebraic varieties over a differential field. One can easily extend many of the results of this paper to the projective case, but we do not address this directly. 

The \emph{type} (in the sense of model theory) of a finite tuple of $ \mc U$ over $k$ is the collection of all first order formulae with parameters from $k$ which hold of the tuple; for a given tuple $\bar a,$ we write $tp (\bar a /k)$. A realization of a type $p$ over $k$ (we write $p \in S(k)$) is simply a tuple from a field extension satisfying all of the first order formulae in the type. As $DCF_{0,m}$ has \emph{quantifier elimination}, we have a correspondence between types and prime differential ideals and differential algebraic varieties. Given a type $p \in S(k),$ we have a corresponding (prime) differential ideal via
$$p \mapsto I_p = \{ f \in k\{y \} \, | \, f(y)=0 \in p \}.$$

The corresponding variety is simply the zero set of $I_p.$ For the reader not acquainted with the language of model theory, the type of a tuple $b$ over $k$ corresponds to the isomorphism type of the differential field extension $k \langle b \rangle /k$ with specified generators. 

When $R$ is a $\Delta$-ring and $S \subseteq R$, by $[S]$ we mean the $\Delta$-ideal generated by $S$ in $R$. When $f$ is an element in a differential polynomial ring, by $V(f)$, we mean the zero set of the differential polynomial. When we think of $f$ as a polynomial in some particular ring, we will write $Z(f)$ for the zero set. Similar notation applies to ideals or sets of elements in (differential) rings. 

Let $\Theta$ be the free commutative monoid generated by $\Delta.$ For $\theta \in \Theta,$ if $\theta=\delta_1^{\alpha_1} \ldots \delta_m^{\alpha_m},$ then $ord(\theta)=\alpha_1+ \ldots + \ldots + \alpha_m.$ The order gives a grading on the monoid $\Theta$. We let $$\Theta (s)=\{ \theta \in \Theta: \, ord(\theta) \leq s \}.$$ Let $k$ be an arbitrary differential field. 

\begin{thm}\label{combinatorialprep} (Theorem 6, page 115, \cite{KolchinDAAG})
Let $\eta = (\eta_1 , \ldots , \eta_n ) $ be a finite family of elements in some extension of $k.$ There is a numerical polynomial $\omega_{\eta /k} (t)$ with the following properties. 
\begin{enumerate} 
\item For sufficiently large $t \in \m N,$ the transcendence degree of $k ((\theta \eta_j)_{\theta \in \Theta(t), \, 1 \leq j \leq n })$ over $k$ is equal to $\omega_{\eta /k} (t).$ 
\item $deg (\omega_{\eta /k }(t)) \leq m$  
\item One can write $$\omega_{\eta / k }(t) = \sum_{o \leq i \leq m} a_i \binom{t+i}{i}$$ In this case, $a_m$ is the \emph{differential transcendence degree} of $k \langle \eta \rangle$ over $k.$ 
\end{enumerate}
\end{thm}

\begin{defn} When $V$ is a differential algebraic variety over $K$ and $\bar b \in V$ is a generic point over $K$, we define $dim(V) $ to be the $\Delta$-transcendence degree of $K \langle \bar b \rangle$ over $K$.
\end{defn} 

\begin{defn} 
The polynomial from the theorem is called the \emph{Kolchin polynomial} or the \emph{differential dimension polynomial}. Let $p \in S(k).$ Then $\omega_p(t):=\omega_{b/k}(t)$ where $b$ is any realization of the type $p$ over $k$. 
\end{defn}

\begin{defn}
Suppose that $p$ and $q$ are types such that $q$ extends $p$. This means $p \in S(k)$ for some differential field $k$ and $q \in S(K)$ where $K$ is a differential field extension of $k$; further, as a sets of first order formulae $p \subset q.$ 
\end{defn}
In this case, any realization of $q$ is necessarily a realization of $p$ when we consider only the formulae which are over $k$. 

\begin{defn}
Let $q$ extend $p$. We say $q$ is a \emph{nonforking extension} of $p$ if $$w_p(t) =\omega_q(t).$$ Note that the Kolchin polynomial on the left is being calculated over $k$ and the Kolchin polynomial on the right is being calculated over $K$. 
\end{defn}

\begin{defn} Let $p$ be a type. Then, \begin{itemize} 
\item $RU(p) \geq 0$ just in case $p$ is consistent. 
\item $RU(p) \geq \beta$, where $\beta$ is a limit just in case $RU(p) \geq \alpha$ for all $\alpha < \beta.$ 
\item $RU(p) \geq \alpha+1$ just in case there is a forking extension $q$ of $p$ such that $RU(q) \geq \alpha.$  
\end{itemize}
\end{defn}

\begin{rem} The last three definitions are specific instances of model theoretic notation in the setting of differential algebra; for the more general definitions, see \cite{GST}. Our development of this is rather nonstandard; normally, forking is defined in a much more general manner. That forking specializes to the above notion in differential algebra requires proof, but is a natural consequence of the basic model theory of differential fields, see \cite{McGrail}. For the differential algebraist, a forking extension of $\bar a$ can be regarded as specialization which occurs after base change to a larger ring for which the locus of $\bar a$ over the larger ring does not descend to the smaller ring. In other words, Lascar rank gives information regarding the structure of differential specializations of a tuple after base change. Specific instances of calculations of Lascar rank in this setting can be found in various sources \cite{MneqU,Jindecomposability,Freitag2014350,sanchez2013contributions,freitag2012model}. 
\end{rem}

When considering model theoretic ranks on types like Lascar rank (denoted $RU(p)$) we will write $RU(V)$ for a definable set (whose $\Delta$-closure is irreducible) for $RU(p)$ where $p$ is a type of maximal Lascar rank in $V$. We should note that some care is required, since the model theoretic ranks are not always invariant under taking $\Delta$-closure. See \cite{FGenerics,freitag2012model} for an example. 

We will be using Lascar rank at various points, and remind the reader of the following result, which we use throughout the paper: 
\begin{prop} \cite{McGrail} \label{MGFACT} Let $b$ be a tuple in a differential field extension of $k.$ Then $$dim(b/k)=n \text{ if and only if } \omega^m \cdot n \leq RU(tp(b/k)) < \omega^m \cdot (n+1)$$
\end{prop} 

We will also require a differential notion of specializations:

\begin{defn} Let $\Delta = \{ \delta_1 , \ldots , \delta_m \}$. Let $\Delta ' = \{ \delta_1 ' , \ldots , \delta_m ' \} $. A homomorphism $\phi$ from $\Delta$-ring $(R, \Delta)$ to $\Delta '$-ring $(S, \Delta ')$ is called a differential homomorphism if for each $i,$  $\phi \circ \delta_i = \delta_i ' \circ \phi.$ When $R$ is an integral domain and $S$ is a field, then such a map is called a \emph{$\Delta$-specialization.}
\end{defn}


The following proposition is proved in a constructive manner in \cite[Theorem 2.16]{GDIT}; and an analogous proof works in the partial case. 

\begin{prop}\label{special} Let $\bar u =(u_1, \ldots , u_r ) \subset \mc U$ be a set of $\Delta$-K independent differential transcendental elements. Let $\bar y =(y_1 , \ldots , y_n)$ be a set of differential indeterminants. Let $P_i(\bar u , \bar y) \in K \{ \bar u , \bar y \}$ for $i=1, \ldots , n_1.$  Suppose $\phi: K \{\bar y \} \rightarrow \mc U$ be a differential specialization into $\mc U$ such that $\bar u$ is a set of $\Delta$-transcendentals over $K \langle \phi (\bar y ) \rangle.$ 
Suppose that $P_i (\bar u , \phi (\bar y))$ are (as a collection), $\Delta$-dependent over $K \langle \bar u \rangle.$ Then let $ \psi$ be a differential specialization from $K\langle \bar u \rangle \rightarrow K.$ The collection $\{ P_i(\psi (\bar u ), \phi (\bar y )) \}_{i=1, \ldots , n_1}$ are $\Delta$-dependent over $K.$ 
\end{prop}

\section{Intersections} \label{section3}
In this section we develop an intersection theory for differential algebraic varieties with generic $\Delta$-polynomials. The influence of \cite{GDIT} for proving statements about irreducibility over specific differential fields is obvious; we have adapted their techniques to the partial differential setting. Our arguments about dimensions of intersections were done earlier from a more model-theoretic point of view. 

\begin{defn}
In $\m A^n,$ the differential hypersurfaces are the zeros of a $\Delta$-polynomial of the form 
$$a_0+\sum a_i m_i$$
where $m_i$ are differential monomials in $\mc F \{y_1, \ldots ,y_n \}.$ For convenience, in the following discussion, we do not consider $1$ to be a monomial. 
A \emph{generic $\Delta$-polynomial} of order $s$ and degree $r$ over $K$ is a $\Delta$-polynomial 
$$f=a_0+\sum a_i m_i$$
where $m_i$ ranges over all differential monomials of order less than or equal to $s$ and degree less than or equal to $r$ (i.e. monomials of $\Theta (s) (\bar y)$ of degree at most $r$) and $ (a_0, a_1, \ldots , a_n)$ is a tuple of independent $\Delta$-transcendentals over $K.$ A \emph{generic $\Delta$-hypersurface} of order $s$ and degree $r$ is the zero set of a generic $\Delta$-polynomial of order $s$ and degree $r.$ When $f$ is given as above, we let $ a_f$ be the tuple of coefficients of $f.$ Throughout, we adopt the notation $\bar a_f =a_f \backslash \{a_0\}.$ 
\end{defn} 


The next lemma is proved in the ordinary case in \cite[Lemma 3.5]{GDIT}. The proof in this case works similarly, assuming that one sets the stage with the proper reduction theory in the partial case. One might notice that Lemma 3.5 of \cite{GDIT} has a second portion. For now, we will concentrate only on the irreducibility of the intersection. Necessary and sufficient conditions for the intersection to be nonempty will be given later.

\begin{lem}\label{3.5}
Let $\mf I$ be a prime $\Delta$-ideal in $K \{ \bar y\}$ with differential transcendence degree $d$ and let $f=y_0 +\sum_{i} a_i m_i$ be a generic degree $d_1$ and order $h$ differential polynomial ($\bar a$ is a tuple of $\Delta$-transcendentals over $K$). Then $\mf I_0=[\mf I,f ]$ is a prime $\Delta$-ideal of $K \langle \bar a_f \rangle \{\bar y , y_0 \}.$ Further, $ \mf I _0 \cap K \langle \bar a_f \rangle \{y_0 \}  \neq 0 $ if and only if $V$ has dimension zero.  
\end{lem}
\begin{proof} Let $\bar b = (b_1, \ldots , b_{n})$ be a generic point of $V(I)$ over $K$ such that $\bar b$ is independent from $\bar a$ over $K.$ In model theoretic terms, $\bar b \ind_K \bar a$. Let $f=y_0+\sum_{i=1}^n a_i m_i(\bar y)$. Consider the tuple $(b_1, \ldots, b_{n_1}, -\sum_i a_i m_i(\bar b)).$ 

Let $\mf I_0=[\mf I,f].$ We show irreducibility of the variety $V(\mf I_0)$ in $\m A^{n+1}$ via showing that it is the Kolchin closure of $(b_1, \ldots, b_{n}, -\sum_i a_i m_i( \bar b))$ over $K.$ Since only irreducible sets over $K$ have $K$-generic points in the Kolchin topology, this will complete the proof (said another way, being the locus over $K$ of a tuple in a differential field extension is precisely equivalent to being an irreducible $\Delta$-$K$-closed set). 

Suppose $g$ is a $\Delta$-polynomial in $K \langle \bar a_f \rangle \{\bar y , y_0 \}$ which vanishes at $(b_1, \ldots, b_{n}, -\sum_i a_i m_i(\bar b)).$   Fix a ranking so that $y_0$ is the leader of $f.$ Then reducing $g$ with respect to $f$ gives some $g_0$ (which is equivalent to $g$ modulo $f$). This $g_0$ must be in $K \langle \bar a_f \rangle \{y_1, \ldots , y_n \}$. Of course, since $\bar b$ is generic for $\mf I,$ we must have that $g_0 \in K \langle \bar a_f \rangle \cdot \mf I.$ But then $g \in \mf I_0$ and the claim follows. 

The natural map of varieties $ V \rightarrow V ( \mf I _0)$ is dominant onto the $y_0$ coordinate if and only $dim (V) >0$, so $ \mf I _0 \cap K \langle \bar a_f \rangle \{y_0 \}  \neq 0 $ if and only if $V$ has dimension zero. 
\end{proof} 

Now we turn towards establishing necessary and sufficient conditions for the intersection to be nonempty when we relax the sorts of intersections under consideration. In the case that the intersection is nonempty, we calculate the differential transcendence degree. 

\begin{lem}\label{generichyp1} Suppose that $V$ is a differential algebraic variety such that $RU(V/K)< \omega^m.$ Then $V \cap V( f (\bar x)) =\emptyset$ for any generic differential polynomial $f(\bar x).$ 
\end{lem}
\begin{proof} This was originally proved in \cite[Theorem 1.7]{Pongembedding} in the ordinary case, and was reproved in \cite{GDIT} in the ordinary case. The proof in the partial case can be found in \cite[Proposition 4.1]{Freitag2014350}.
\end{proof}


\begin{lem}\label{generichyp2} Suppose that $V$ is a differential algebraic variety embedded in $\m A^n$ and that $V$ is of dimension $d \geq 1$.  If $f(\bar x)$ is a generic differential polynomial, then $V \cap V (f) \neq \emptyset$ and $dim( V \cap V(f)) =d-1$. 
\end{lem}
\begin{proof} 
Let $V = V ( \mf I)$; we will following the general notation of \ref{3.5}, with $f = a_0 + \sum _i a_i m_i (\bar y)$. That is, we will move back and forth between viewing the degree $0$ and order $0$ coefficient of $f$, $a_0$, as an element of a field extension of $K$ and as an indeterminant. We use the notation $f_0  = \sum _i a_i m_i (\bar y).$ 
Let $ \bar b $ be a realization of the generic type of $V$ over $K$. Specifically, we will think of $ \mf I_0 \subseteq K  \langle \bar a_f \rangle \{\bar y , a_0 \}$ as in \ref{3.5} and $ \mf I_1 = [ \mf I , f ] \subseteq K  \langle a_f \rangle \{\bar y \}$. 

Reorder the coordinates if necessary so that $b_1, \ldots , b_d$ are a $\Delta$-transcendence basis for the $\Delta$-field extension generated by $\bar b $ over $K$. Because $V$ is isomorphic to $V ( \mf I_0 )$, $dim (V ( \mf I_0)) = d$. Since $ \mf I \subset \mf I _1$, each of $ y_i$ for $i >d$ is $\Delta$-dependent with $y_1 , \ldots ,y_d$ modulo $\mf I$ (and thus $\mf I_1$).  Now $ a_0 , y_1 , \ldots , y_d$ are $ \Delta$-dependent modulo $ \mf I_0$ and so $y_1 , \ldots , y_d$ are $\Delta$-dependent modulo $\mf I_1$. Thus $dim (V (  \mf I_1 )) \leq d-1$. 

Now supopose that $y_1 , \ldots , y_{d-1}$ are $\Delta$-dependent modulo $\mf I_1$; then there is some nonzero element $p ( y_1 , \ldots , y_{d-1}) \in \mf I_1$. By clearing denominators one, can take $p \in K \{ \bar a_f,y_1, \ldots ,y_{d-1}, a_0 \}$. Then $p ( \bar a_f , b_1 , \ldots , b_{d-1},  - f_0 (\bar b)) = 0 .$  Now specialize $a_d,$ the coefficient of $y_d$ in the generic differential hypersurface to $-1$ and specialize all other $a_i \in \bar a_f$ to $0$. But then $b_1 , \ldots ,b_d$ are dependent over $K$ by \ref{special}, a contradiction to the assumption that $V$ has dimension $d$. 

\end{proof}

\begin{lem}\label{3.6} 
Let $\mf I$ be a prime $\Delta$-ideal in $K \{ y_1 , \ldots , y_n \}.$ Let $f=a_0+ \sum_{i} a_i m_i (\bar y) $ give a generic hypersurface. Then $\mf I_1 =[ \mf I, f ]$ is a prime $\Delta$-ideal in $K \langle a_0 , a_1 , \ldots, a_n \rangle \{ y_1 , \ldots , y_n \}.$ 
\end{lem} 
\begin{proof}
First, suppose that the dimension of $V$ is at least one. Then by Lemma \ref{generichyp2}, $V(I) \cap V(f) \neq \emptyset.$ Recall the notation of $\mf I_0$ from Lemma \ref{3.5}. We will show that $\mf I_1 \cap K \langle a_1, \ldots , a_n \rangle \{y_1, \ldots , y_n , y_0\} = \mf I_0.$ Suppose that we have $g,h \in K \langle a_1 , \ldots , a_n, y_0 \rangle \{y_1, \ldots , y_n \}$ such that $g \cdot h \in \mf I_1.$ Since we are taking a field extension over $K$, the coefficients of the differential polynomials might involve differential rational functions in $a_1, \ldots , a_n, y_0$ over $K.$ This is easily dispensed with since if we multiply by suitable differential polynomials in $a_0 , \ldots , a_n, y_0$ over $K,$ we will get $g,h \in K \{ a_1 , \ldots , a_n, y_0, y_1, \ldots ,y_n\}$ such that $g \cdot h \in \mf I_0.$ But, $\mf I_0$ is prime by Lemma \ref{3.5}. So, we have a contradiction and $\mf I_1$ is prime. Further, we can see (again, simply by clearing denominators) that $\mf I_1$ lies over $\mf I_0,$ when we regard $\mf I_0$ as an ideal of $R \{y_1 , \ldots , y_n \}$ where $R=K \langle a_1, \ldots ,a_n \rangle \{y_0 \}.$  

In the case that $dim (V) = 0$, $ \mf I _0 \cap K \langle \bar a _f \rangle \{a_0 \} \neq 0$ by Lemma \ref{3.5}, so it is easy to see that $ \mf I_1$  must be the unit ideal. 
\end{proof}


\begin{lem}\label{3.6.1} Following the notation of the previous lemma, let $d>0.$ Let $f$ be order $h$ and degree $d_1$.
Then, $$\omega_{V([\mf I ,f])/K \langle a_0 , a_1 ,\ldots , a_n \rangle}(t)=\omega_{V( \mf I)/ K }(t) -\binom{t+m-h}{t}.$$ 
\end{lem} 

\begin{proof} In this proof, we associate a differential algebraic variety $V$ naturally with its \emph{prolongation sequence} (for complete details, see \cite{MPSarcs2008}). Briefly, recall, the data of a prolongation sequence is the sequence of algebraic varieties: 
$$V_l = \{(\theta \bar x ) \,| \, x \in X(\mc U) , \,  \theta \in \Theta (l)\}^{cl} \subseteq \m A^{n \cdot \binom{l+m}{m}}, $$
where $(-)^{cl}$ denotes Zariski closure and the coordinates are ordered by the canonical orderly ranking induced by taking $\delta_i < \delta_j$ when $i <j$. It is a fact that the sequence $V_l$ determines $V$; by Noetherianity of the Kolchin topology, a finite subsequence determines $V$. Given a sequence of algebraic varieties $(V_l) _{l \in \m N}$ with $V_l \subseteq \m A^{n \cdot \binom{l+m}{m}}$, we call the sequence a prolongation sequence if for all $k >l$, the natural projection map $V_ k \rightarrow V_l$ dominant and the variety $V_{l+1}$ satisfies the differential relations forced by $V_l$. In the notation of \cite[page 7, preceeding Proposition 2.5]{MPSarcs2008}, $V_{l+1} \subseteq \tau (V_l)$. There is a bijective correspondence between differential algebraic varieties and prolongation sequences. $V$ is irreducible over $K$ if and only if the all of the  varieties in the corresponding prolongation sequence are irreducible over $K$. 


For large enough values of $l$, the dimension of $V_l$ is given by the value of the Kolchin polynomial of $V$. We will show that $$V_l \cap \bigcap _{\theta \in \Theta (l-h)}  (Z(\theta (f)))$$ is a prolongation sequence. Once this fact is established, it is clear that it must be the prolongation sequence corresponding to $V \cap V(f)$ (since each the given algebraic relations clearly hold on $V \cap V(f)$).
Consider the differential algebraic variety $W \subseteq \m A^{n+1}$ given by $\m I _0$ as in Lemma \ref{3.5} above.
As differential algebraic varieties, $W$ and $V$ are isomorphic, by the obvious maps. 
The prolongation sequence associated with $W$ is given by $W_l=V_l \times \m A^1_l$ when $ l < h$ and 
$$W_l=(V_l \times \m A^1_l ) \cap \bigcap _{\theta \in \Theta(l-h) } V(\theta(f_{y_0})),$$
when $l\geq h$,  where $f_{y_0}$ is the differential polynomial $f$ with $y_0$ in place of $a_0$. To verify that this is a prolongation sequence, we need only determine that the maps $W_{l+1} \rightarrow W_l$ are dominant (since the second condition is obvious from the definition of the sequence). Then since the relation $f_{y_0}=0$ holds on $W$, the sequence must be the prolongation sequence associated with $W$. 
When $l+1 <h$, this follows simply from the fact that $V_l$ forms a prolongation sequence, noting that $W_l = V_ l \times \m A^1_l$.  

When $l+1 \geq h$, $W_l$ is a subvariety of $V_l \times \m A^1 _l $ determined by the zero set of $\theta (f_{y_0}) =0$ for each $\theta \in \Theta (h-l)$. So, $W_l$ is a subvariety of $$W_l =  \left( \left(V_l \times \m A^{\binom{l+h+m}{m}} \right) \cap  W_l ' \right)\times A^{\binom{l+m}{m} - \binom{l-h+m}{m}},$$  where 

$$W_l' \subseteq V_l \times \m A^1 _{l-h}$$  which is given the vanishing of $\theta (f_{y_0}) =0$ for each $\theta \in \Theta (l-h)$. Since $\theta (f_{y_0})$ is linear in $\theta (y_0)$, $W_l'$ is the graph of a function $V_l \rightarrow \m A^1_{l-h}.$

By our assumption $dim(V)>1$, for some $j$, $ \{\theta \circ \delta_m^h y_j \, | \, \theta \text{ of order $l-h$}\}$ are independent transcendentals over $K \langle \bar a_f \rangle (\{\eta y_i \, | \, \eta \in \Theta (l)  ,\, i \neq j \} \cup \{ \eta y_j \, | \, \eta <  \delta_1^{l-h} \delta_m ^h  \} )$ where we note that $$\{ \eta y_j \, | \, \eta <  \delta_1^{l-h} \delta_m ^h  \} = \{ \eta y_j \, | \, \eta \in \Theta (l)   \} \setminus  \{\theta \circ \delta_m^n y_j \, | \, \theta \text{ of order $l-h$}\}.$$

But, on $W_{l+1},$ for $\theta$ of order $l+1 -h$, $\theta y_0$ is linearly dependent with $\theta \delta_m ^h (y_j)$ over $K \langle \bar a_f \rangle (\{\eta y_i \, | \, \eta \in \Theta (l)  ,\, i \neq j \} \cup \{ \eta y_j \, | \, \eta <  \delta_1^{l-h} \delta_m ^h  \} ).$ 

Thus on $W_{l+1}$, $\{\theta (y_0) \, | \, \theta  \text{ of order $l+1-h$}\}$ are transcendental over $$K \langle \bar a_f \rangle (\{\eta y_i \, | \, \eta \in \Theta (l+1)  ,\, i \neq j \} \cup ( \{ \eta y_j \, | \, \eta \in \Theta (l+1)   \} \setminus  \{\theta \circ \delta_m^n y_j \, | \, \theta \text{ of order $l+1-h$}\}) ).$$ Thus, onto the coordinates indexed by $\theta y_0 $ for $\theta$ of order $l+1-h$, the image of $W_{l+1} $ in $W_l$ is dense. From this it is easy to see that the map $W_{l+1} \rightarrow W_l$ must be dominant. 

The sequence of varieties $W_l = V_l \cap \bigcap _{\theta \in \Theta (l-h)}  V( \theta (f))$ is the (generic) fiber above $(V_0)_l$ (for $y_0 =a_0$).  It is a fact following from the Noetherianity of the Kolchin topology that a prolongation sequences is determined by some finite subsequence. Verifying the dominance of the projection maps in a sequence $(V_l)$ for $l$ bounded by some $s$ is an open condition on the coordinates above $y_0$. That is, the condition that these finitely many maps be dominant is constructible in $\{\theta (y_0) \, | \, \theta \in \Theta (\alpha)\}$, over $K$ and since that property holds for $(V_0)_l$, it must hold for the generic fiber of the family above $y_0$ (dominance must hold on some Zariski open subset of $\{\theta (y_0) \, | \, \theta \in \Theta (\alpha)\}$). Thus $$V_l \cap \bigcap _{\theta \in \Theta (l-h)} Z(\theta (f))$$ is the prolongation sequence of $V \cap V(f)$. 

For the calculation of the Kolchin polynomial, note that both $V_l \cap \bigcap _ {\theta \in \Theta (l-h)} Z(\theta (  f))$ and $W_l$ are sequences of irreducible varieties and the surjective map $W_l \rightarrow V_l \cap \bigcap _ {\theta \in \Theta (l-h)}  Z(\theta ( f))$ has fiber dimension $\binom{l+m}{m}$ since $y_0$ is a differential transcendental over $K$. 

The dimension of $W_l$ is given by $dim (V_l ) + \binom{l+m}{m}- \binom{l+m-h}{m}$ because above the coordinate $y_0$, $\binom{l+m}{m} $ is the number of coordinates $\Theta (l ) (y_0)$, and $ \binom{l+m-h}{m}$ is the number of equations (linear in $\Theta (l ) (y_0)$) which appear in the definition of $W_l$; the equations are independent by the fact that $dim(V)>0$, and so $V_0$ is Kolchin dense in the $y_0$ coordinate.



\end{proof}

Putting together the previous results of the section, and restating the theorem in the form we will apply it later, we have:
\begin{thm}\label{GDIT} Let $V$ be a Kolchin-closed (over $K$) subset of $\m A^n$ with differential transcendence degree $d.$ Let $H$ be a generic (with respect to $K$) differential hypersurface of some degree and order $h$ with coefficients given by $\bar a$. Then $V \cap H$ is irreducible over $K \langle \bar a \rangle$. 
In the case that $d=0,$ $V \cap H =\emptyset.$ If $d>0$, then the Kolchin polynomial of $V \cap H$ is given by $$\omega_{V \cap H/K \langle a_H \rangle } (t) = \omega_{V/K}(t) -\binom{t+m-h}{m}.$$
\end{thm} 

One key point to notice is that $$\binom{t+m-h}{m} =\binom{t+m}{m} - \sum_{i=0} ^{h-1} \binom{t+m-1-i}{m-1}$$ as long as $h>0$ and $t$ is sufficiently large, and that under these circumstances, $$\sum_{i=0} ^{h-1} \binom{t+m-1-i}{m-1}$$ is a \emph{positive} integer. In the special case that $m=1,$ this integer is $h$; meaning the previous theorem is a generalization of the following theorem, proved in \cite[Theorem 3.13]{GDIT} when $K$ is an ordinary differential field: 
\begin{thm}\label{ordinary} Let $ \mf I$ be a prime $\delta$-ideal in $K\{ \bar y \}$ with Kolchin polynomial $(t+1) d +c$. Let $f$ be a generic $\delta$-polynomial of order $h$ and degree $d$. 
Then $ \mf I _ 1 = [\mf I , f] $ is a prime $\delta$-ideal in $K \langle a_f \rangle \{ \bar y \}$ with Kolchin polynomial $(t+1) (d-1) +c+h$
\end{thm}

\begin{rem} 
Note that we are considering the Kolchin topology over a \emph{specific field}, and not its algebraic closure. Irreducibility over the algebraic closure of the differential field of definition is what we call geometric irreducibility. We have not proved this yet, nor do the authors of \cite{GDIT} in the ordinary setting. In fact, at least one additional hypothesis is necessary for that result: if the hypothesis were \emph{purely in terms of dimension}, we would have to restrict to the situation $d \geq 2.$ After all, take any degree $d_1>1$ plane curve. This curve meets the generic hypersurface of degree $d_2$ in precisely $d_1\cdot d_2$ points, so the intersection is not irreducible over any algebraically closed field. In fact, in the next section, we show that this is the only potential problem. 
\end{rem} 

\section{Geometric irreducibility} \label{section4}

Before discussing geometric irreducibility, we will require some results about the Kolchin polynomials of prime differential ideals lying over a fixed prime differential ideal in extensions. 

\begin{prop}\label{lyingover} (\cite{KolchinDAAG} pg131, proposition 3, part b) Let $\mf p$ be a prime differential ideal in $K \{ y_1 , \ldots , y_n \}$ and let $F$ be a differential field extension of $K.$ Then $F \mf p$ has finitely many prime components in $F \{ y_1 , \ldots , y_n \}$. If $\mf q$ is any of the components, then a generic type of the variety $V(\mf q )$ has the same Kolchin polynomial as the generic type of $V(\mf p ).$ 
\end{prop} 

\begin{rem}\label{lyingover1}
In model theoretic terms, the generic types of the components $V(\mf p_1) ,\ldots , V(\mf p_n)$ of $V( \mf p)$ are each nonforking extensions of the generic type of $V(\mf p)$. Assuming that the base field $K$ is algebraically closed would ensure that the generic type of $V(\mf p)$ is stationary; consequently $F \mf p$ is a prime differential ideal for any field extension $F$ of $K.$
\end{rem}

Recall the following definition given in the introduction:
\begin{defn} An affine differential algebraic variety, $V$ over $K$, is geometrically irreducible if $I(V/K')$ is a prime differential ideal for \emph{any} $K'$, a differential field extension of $K$.
\end{defn} 

\begin{rem} 
The previous remark shows that it is enough to consider irreducibility over $K^{alg}$, the algebraic closure of $K$. To put geometric irreducibility in the language of differential schemes, if $V=\Delta Spec(K \{x_1, \ldots , x_n \} / \mf p)$ where $\mf p$ is a prime differential ideal, then $V$ is geometrically irreducible if its base change $\Delta Spec(K \{x_1, \ldots , x_n \} / \mf p) \times _{\Delta Spec (K)} \Delta Spec( K^{alg})$ is irreducible.
\end{rem}

\begin{thm} Let $V$ be a geometrically irreducible Kolchin-closed over $K$ subset of $\m A^n$ with Kolchin polynomial $\omega_V(t)>\binom{t+m}{m}.$ Let $H$ be a generic hypersurface of some degree $d_1$. 
Then $V \cap H$ is geometrically irreducible and $\omega_{H \cap V}(t)=\omega_{V/K}(t) -\binom{t+m}{m}$. 
\end{thm}
\begin{proof}
Consider the the differential algebraic variety $W= \{(v_1, v_2, \beta) \, | \, v_i \in V, \, v_i \in H_\beta \} \subseteq V \times V \times \m A^n$ where $H_\beta$ is the hypersurface given by $\sum \beta_i m_i=1,$ where the sum ranges over all monomials in $\bar x$ of degree bounded by $d_1$. 
It might be the case that $W$ is reducible in the Kolchin topology, but we will not be concerned with this issue specifically. 

Consider $V \cap H_\beta$. When $\beta$ is generic over $K$, we know that $V \cap H_\beta$ is irreducible over $K\langle \beta \rangle,$ so by the Proposition \ref{lyingover}, all of the components of $V$ over the algebraic closure of $K\langle \beta \rangle$ have Kolchin polynomial equal to $\omega_{V \cap H_\beta /K \langle \beta \rangle }(t).$ If $V \cap H_\beta$ has more than one component, then $W$ has more than one component with Kolchin polynomial at least $$2 \cdot \omega_{V/K}(t)  + \left(\binom{n+d_1}{d_1}-3\right) \cdot  \binom{m+t}{t}.$$ 

To see this, first note that the length of the tuple $\beta$ is $(\binom{n+d_1}{d_1} -1)$, the number of monomials of degree bounded by $d_1$ in $n$ variables, excluding $1$. The Kolchin polynomial of a generic $\beta$ over $K$ is given by $\left(\binom{n+d_1}{d_1}-1\right) \cdot  \binom{m+t}{t}$. The Kolchin polynomial of two independent generic points $(v_1,v_2)$ on $V \cap H_ \beta$ is given by $2\left( \omega_{V/K}(t) -\binom{t+m}{m}\right)$. Thus by Sit's lemma \cite[Lemma 2.9]{Jindecomposability}, the Kolchin polynomial of the tuple $(v_1,v_2,\beta)$ is at least $$2 \cdot \omega_{V/K}(t)  + \left(\binom{n+d_1}{d_1}-3\right) \cdot  \binom{m+t}{t}.$$ 

Suppose there is more than one component of $V \cap H _\beta $ over $K \langle \beta \rangle^{alg}$.  Then there is more than one option for a complete type on $V \cap H_ \beta$ of rank $\omega_{V/K}(t) -\binom{t+m}{m}$, so there is more than one option for the type $(v_1, v_2, \beta)$ with $v_1,v_2$ generic and independent on $V \cap H_\beta$ over $K \langle \beta \rangle$, depending on if $v_1$ and $v_2$ are in the same component of $V \cap H_\beta$. Now, we only consider components of $W$ with Kolchin polynomial at least $2 \cdot \omega_{V/K}(t)  + \left(\binom{n+d_1}{d_1}-3\right) \cdot  \binom{m+t}{t}.$

Suppose $v_1$ and $v_2$ are points on $V,$ and $\beta$ is generic subject to the condition that $H_\beta$ contains $v_1, v_2$. If $v_1 \neq v_2$, then we claim that $$\omega _{v_1, v_2, \beta  / K} (t) = \omega _{v_1,v_2/K}(t)+\left(\binom{n+d_1}{d_1}-3\right) \cdot  \binom{m+t}{t}.$$ To see this, simply note that for $v_1 \neq v_2$, we get two independent linear conditions on $\beta$. 

The only way that $$\omega _{v_1,v_2/K}(t)+\left(\binom{n+d_1}{d_1}-3\right) \cdot  \binom{m+t}{t} \geq 2 \cdot \omega_{V/K}(t)  + \left(\binom{n+d_1}{d_1}-3\right) \cdot  \binom{m+t}{t}$$ is for $v_1,v_2$ to be independent generic points on $V$, in which case, equality holds. 

By similar analysis, $v_1 = v_2, $ then $\omega _{v_1,v_2,\beta /K} (t)  = \omega _{ v_1 /K} (t) + \left(\binom{n+d_1}{d_1}-2\right) \cdot  \binom{m+t}{t}.$ 

Since $\omega _ V (t) > \binom{t+m}{m}$, $$\omega _{ v_1 /K} (t) + \left(\binom{n+d_1}{d_1}-2\right) \cdot  \binom{m+t}{t} < 2 \cdot \omega_{V/K}(t)  + \left(\binom{n+d_1}{d_1}-3\right) \cdot  \binom{m+t}{t}.$$
So there is a unique type on $W$ of rank $2 \cdot \omega_{V/K}(t)  + \left(\binom{n+d_1}{d_1}-3\right) \cdot  \binom{m+t}{t}.$

By our earlier arguments, there is a unique component of $V \cap H_\beta$ over $ K \langle \beta \rangle ^{alg}$ with Kolchin polynomial $\omega_V -\binom{t+m}{m}.$ But, by Proposition \ref{lyingover}, we know any of component of $V \cap H_\beta$ must have Kolchin polynomial $\omega_V -\binom{t+m}{m}.$ So, $V \cap H_\beta$ is geometrically irreducible. 
\end{proof}

In the proof of the previous theorem, we can weaken the assumption that $\omega_V(t) > \binom{t+m }{m}$ to $\omega _V (t) \geq \binom{t+m}{m}$ in the case that the order of the differential hypersurface we consider is greater than $0$. That is, in this case, $V$ might be an algebraic curve. 

\begin{thm} Let $V$ be a geometrically irreducible Kolchin-closed over $K$ subset of $\m A^n$ with Kolchin polynomial $\omega_V(t)>\binom{t+m}{m}.$ Let $H$ be a generic differential hypersurface of order $h$ and degree $d_1$. 
Then $V \cap H$ is geometrically irreducible and $$\omega_{H \cap V}(t)=\omega_V(t) - \binom{t+m-h}{m}.$$ 
\end{thm} 
\begin{proof} The proof is analogous to the previous proof, so we will be brief. The rank calculations are the only appreciable difference. We define $W= \{(v_1, v_2, \beta) \, | \, v_i \in V, \, v_i \in H_\beta \} \subseteq V \times V \times \m A^{n_1}$ where $H_\beta$ is the differential hypersurface given by $\sum \beta_i m_i=1,$ where the sum ranges over all monomials in $\bar x$ of order bounded by $h$ and degree bounded by $d_1$. Note that $$n_1 =\binom{n \cdot \binom{t+m}{m}+d_1}{d_1}-1.$$

Fixing two independent generic points on $V \cap H_\beta$, $v_1$ and $v_2$ and choosing coefficients of the generic differential hypersurface, relative to the condition that $H_ \beta$ contains $v_1$ and $v_2$ gives a tuple with Kolchin polynomial $$2 \omega _V (t) - 2 \binom{t+m-h}{m} + n_1 \binom{t+m}{m}.$$ Again, if there is more than one type on $ V \cap H_\beta$ with Kolchin polynomial $\omega _V (t) -  \binom{t+m-h}{m}$, then there is more than one type on $W$ with Kolchin polynomial $2 \omega _V (t) - 2 \binom{t+m-h}{m} +n_1 \binom{t+m}{m}.$

In general, for any (possibly) non-generic choice of $v_1,v_2 \in V$, if $v_1 \neq v_2, $ the Kolchin polynomial of $(v_1, v_2, \beta)$ is bounded by $\omega _{v_1,v_2/K}(t) + n_1 \binom{t+m}{m} - 2 \binom{t+m-h}{m}$. Thus, the only way for  $$\omega _{v_1,v_2/K}(t) + n_1 \binom{t+m}{m} - 2 \binom{t+m-h}{m} = 2 \omega _V (t) - 2 \binom{t+m-h}{m} + n_1 \binom{t+m}{m}$$ is to choose $v_1, v_2$ independent generics on $V$. 
If $v_1 = v_2$ then the Kolchin polynomial of $(v_1, v_2, \beta)$ is bounded by \begin{eqnarray*} \omega _{v_1/K}(t) +   n_1  \binom{t+m}{m}   -  \binom{t+m-h}{m}     <  2 \omega _V (t) - 2 \binom{t+m-h}{m} + n_1\binom{t+m}{m}.\end{eqnarray*} Thus by reasoning similar to the previous proof, $V \cap H$ is geometrically irreducible. 
\end{proof} 

\begin{rem} There are several notions of \emph{smoothness} in the context of differential algebraic geometry, coming from the differential arc spaces considered in \cite{MPSarcs2008, MSJETS} and from Kolchin's differential tangent spaces \cite{KolchinDAAG}. One can prove that generic intersections preserve any of these notions of smoothness (for more details, see\cite{freitag2012model}). 

\end{rem}

\section{Generic differential hypersurfaces through a given point}

The authors of \cite{GDIT} prove a geometric result which also generalizes to the partial differential setting. Their proof uses differential specializations. Our approach here is rather different, though a proof by suitably generalized methods of \cite{GDIT} is possible. Our proof is shorter, but we are using the machinery of stability theory. 

\begin{thm}\label{geohyphyp} Let $V$ be a differential algebraic variety of dimension $d.$ If the set of $d+1$ independent generic hyperplanes through $\bar a$ intersects $V$, then $\bar a \in V.$  
\end{thm}
\begin{proof}
We note that the hypotheses imply that $\omega^m \cdot d \leq RU(V/K) < \omega^m \cdot(d+1)$ (see \ref{MGFACT}). 
Let $\bar a \notin V.$ First, we will argue the result in the case that $\bar a = (0, \ldots , 0).$ Any hyperplane through the origin is of the the form $\sum c_i y_i=0.$ We assume that the $c_i$ are independent differential transcendentals over $K.$ We denote this hyperplane by $H_ {\bar c}.$ Suppose that $\bar b$ is a generic point on one of the irreducible components of $V \cap H_{\bar c}$ over $K \langle \bar c \rangle.$ If $\bar d$ is a generic point on $V$ over $K,$ and for $I \subseteq \{1, 2, \ldots, n\}$ we have that $D=\{d_i \, | \, i \in I \}$ is a differential transcendence base for the field extension $K \langle \bar d \rangle /K,$ then the same property holds for $\bar b,$ that is $B=\{b_i \, | \, i \in I \}$ is a differential transcendence base for the field extension $K \langle \bar b \rangle /K.$ Thus, $RU(\bar b / K \langle B \rangle)< \omega^m.$ 

Now since $\bar b \in H_{\bar c } \cap V,$ we know that $\sum c_i b_i =0.$ We will bound $RU(\bar b / K \langle \bar c \rangle).$ Since over $K,$ $\bar c$ is an independent differential transcendental, $$RU(\bar c / K \langle \bar b \rangle) + \omega^m \leq  RU(\bar c /K).$$ 
But, then by Lascar's symmetry lemma \cite[chapter 19]{PoizatModel}
$$RU( \bar b / K \langle \bar c \rangle ) + \omega^m  \leq RU(\bar b /K).$$
Thus, the differential transcendence degree of $\bar b /K \langle \bar c \rangle$ is at least one less than that of $\bar d /K.$ 

In the case that $RU(V /K) < \omega^m,$ the above argument using Lascar's symmetry lemma shows that $V \cap H_{\bar c} =\emptyset$. The full result in the case $\bar a$ is the origin follows by induction. 
Now, suppose that $\bar a$ is some point besides $(0, \ldots , 0).$ If so, adjoin $\bar a$ to the field $K$ and consider $K \langle \bar a \rangle.$ A priori, perhaps $V$ is no longer irreducible; if not, arguing about each irreducible component would suffice.  Now, by translating the variety $V$ and the point $\bar a$, one can assume $\bar a = (0 , \ldots , 0).$ 
\end{proof}

Notice that this result, in the case that $d=0$, generalizes \ref{generichyp1}, which is itself a generalization of \cite[Theorem 1.7]{Pongembedding}. Notice a further advantage of our proof is that it is easily adapted to the case of generic differential hypersurfaces, giving the following theorem:

\begin{thm}\label{genhyphyp} Let $V$ be a differential algebraic variety of dimension $d.$ If the set of $d+1$ independent generic differential hypersurfaces (each of arbitrary degree and order) through $\bar a$ intersects $V$, then $\bar a \in V.$  
\end{thm} 

We also note that the previous result and various results of this paper can also be seen to hold under the weaker hypothesis of quasi-genericity \cite[page 17 for the definition]{GDIT} of the differential hypersurfaces. 

\bibliography{/Users/freitagj/Dropbox/Research}{}

\begin{thebibliography}{10}

\bibitem{buium1994geometry}
Alexandru Buium.
\newblock Geometry of differential polynomial functions, {II}: algebraic
  curves.
\newblock {\em American Journal of Mathematics}, pages 785--818, 1994.

\bibitem{FGenerics}
James Freitag.
\newblock Generics in differential fields.
\newblock {\em In Preparation,
  http://math.berkeley.edu/people/faculty/james-freitag}, 2012.

\bibitem{freitag2012model}
James Freitag.
\newblock {\em Model theory and differential algebraic geometry}.
\newblock PhD thesis, University of Illinois, 2012.

\bibitem{Freitag2014350}
James Freitag.
\newblock Completeness in partial differential algebraic geometry.
\newblock {\em Journal of Algebra}, 420(0):350 -- 372, 2014.

\bibitem{Jindecomposability}
James Freitag.
\newblock Indecomposability in partial differential fields.
\newblock {\em Accepted, Journal of Pure and Applied Algebra,
  http://arxiv.org/abs/1106.0695}, 2014.

\bibitem{OmarWilliamJim}
James Freitag, Omar~Le\'on Sanchez, and William Simmons.
\newblock Linear dependence and completeness for differential algebraic
  varieties.
\newblock {\em Under Review, Communications in Algebra}, 2014.

\bibitem{GDIT}
Xiao-Shan Gao, Wei Li, and Chun-Ming Yuan.
\newblock Intersection theory in differential algebraic geometry: Generic
  intersections and the differential chow form.
\newblock {\em Transactions of the American Mathematical Society},
  365(9):4575--4632, 2013.

\bibitem{Hartshorne}
Robin Hartshorne.
\newblock {\em Algebraic Geometry}.
\newblock Graduate Texts in Mathemematics 52. Springer-Verlag, 1977.

\bibitem{hodge1994methods}
W.V.D. Hodge and D.~Pedoe.
\newblock {\em Methods of Algebraic Geometry}.
\newblock Number v. 2 in Cambridge Mathematical Library. Cambridge University
  Press, 1994.

\bibitem{MneqU}
Ehud Hrushovski and Thomas Scanlon.
\newblock Lascar and {M}orley ranks differ in differentially closed fields.
\newblock {\em Journal of Symbolic Logic}, 64:no. 3,1280--1284, 1986.

\bibitem{KolchinDiffComp}
Ellis Kolchin.
\newblock Differential equations in a projective space and linear dependence
  over a projective variety.
\newblock {\em Contributions to analysis (A collection of papers dedicated to
  {L}ipman {B}ers}, pages 195--214, 1974.

\bibitem{KolchinDAAG}
Ellis~R. Kolchin.
\newblock {\em Differential Algebra and Algebraic Groups}.
\newblock Academic Press, New York, 1976.

\bibitem{Marker}
David Marker.
\newblock {\em Model Theory: an Introduction}.
\newblock Springer, Graduate Texts in Mathematics, 217, Second Edition, 2002.

\bibitem{McGrail}
Tracy McGrail.
\newblock The model theory of differential fields with finitely many commuting
  derivations.
\newblock {\em Journal of Symbolic Logic}, 65, No. 2:885--913, 2000.

\bibitem{MPSarcs2008}
Rahim Moosa, Anand Pillay, and Thomas Scanlon.
\newblock Differential arcs and regular types in differential fields.
\newblock {\em Journal fur die reine and angewandte Mathematik}, 620:35--54,
  2008.

\bibitem{MSJETS}
Rahim Moosa and Thomas Scanlon.
\newblock Jet and prolongation spaces.
\newblock {\em Journal de l'Institut de Mathematiques de Jussieu},
  9(2):391--430, 2010.

\bibitem{GST}
Anand Pillay.
\newblock {\em Geometric Stability Theory}.
\newblock Oxford University Press, 1996.

\bibitem{PoizatModel}
Bruno Poizat.
\newblock {\em A course in model theory}.
\newblock Springer, New York, 2000.

\bibitem{Pongembedding}
Wai~Yan Pong.
\newblock Some applications of ordinal dimensions to the theory of
  differentially closed fields.
\newblock {\em Journal of Symbolic Logic}, 65, No. 1:347--356, 2000.

\bibitem{Ritt}
Joseph~Fels Ritt.
\newblock {\em Differential Algebra}.
\newblock Dover Publications, New York, 1950.

\bibitem{sanchez2013contributions}
Omar~Le{\'o}n S{\'a}nchez.
\newblock Contributions to the model theory of partial differential fields.
\newblock {\em arXiv preprint arXiv:1309.6330}, 2013.

\bibitem{SitOnlineNotes}
William Sit.
\newblock Anomaly of intersections in differential algebra.
\newblock {\em Online notes,
  \url{http://www.sci.ccny.cuny.edu/~ksda/PostedPapers/Sit021012.pdf}}.

\end{thebibliography}
\bibliographystyle{plain}

\end{document}